\documentclass[11pt]{article}
\usepackage[scale=0.72]{geometry}
\usepackage{authblk}
\usepackage{blindtext}
\usepackage{graphicx}
\usepackage{amsfonts,amsmath,amsbsy,latexsym,mathrsfs,amsthm,amssymb}
\usepackage[colorlinks=true,linkcolor=purple,citecolor=blue]{hyperref}
\usepackage{enumerate}
\usepackage[section]{placeins}

\usepackage{pgf,tikz}
\usepackage{multicol}
\usetikzlibrary{arrows}
\usepackage{caption}
\usepackage{subcaption}
\usepackage{float}

\usetikzlibrary{patterns}
\newtheorem{thm}{Theorem}[section]
\newtheorem{cor}{Corollary}[section]

\newtheorem{prop}{Proposition}[section]
\newtheorem{defi}{Definition}[section]
\newtheorem{rem}{Remark}[section]
\newtheorem{exa}{Example}[section]
\makeatletter
 \let\c@cor\c@thm
 \let\c@lem\c@thm
 \let\c@prop\c@thm
 \let\c@defi\c@thm
 \let\c@rem\c@thm
 \let\c@exa\c@thm
\makeatother

\def \cA{{\mathcal A}}
\def \cB{{\mathcal B}}
\def \cC{{\mathcal C}}
\def \cD{{\mathcal D}}

\def \cM{{\mathcal M}}

\def \cR{{\mathcal R}}

\def \cU{{\mathcal U}}

\def \B{\mathbb{B}}
\def \C{\mathbb{C}}
\def \D{\mathbb{D}}

\def \L{\mathbb{L}}
\def \M{\mathbb{M}}

\def \P{\mathbb{P}}

\def \R{\mathbb{R}}
\def \S{\mathbb{S}}

\usepackage{setspace}

\begin{document}

\title{The manifold of polygons degenerated\\to segments}
\date{}
\author[1]{Manuel A. Espinosa-Garc\'ia \thanks{esgama@matmor.unam.mx}}
\author[2]{Ahtziri Gonz\'alez \thanks{ahtziri.lemus@umich.mx}}
\author[3]{Yesenia Villicaña-Molina \thanks{yesenia.villicana.molina@gmail.com}}
\affil[1]{\small Posgrado Conjunto en Ciencias Matem\'aticas, Universidad Michoacana de San Nicol\'as de Hidalgo - Universidad Nacional Aut\'onoma de M\'exico}
\affil[2]{Facultad de Ingenier\'ia El\'ectrica, Universidad Michoacana de San Nicol\'as de Hidalgo}
\affil[3]{Centro de Ciencias Matem\'aticas, Universidad Nacional Autónoma de México}
\maketitle

\vspace{-0.5in}

\abstract{In this paper we study the space $\L(n)$ of $n$-gons in the plane degenerated to segments. We prove that this space is a smooth real submanifold of $\C^n$, and describe its topology in terms of the manifold $\M(n)$ of $n$-gons degenerated to segments and with the first vertex at 0. We show that $\M(n)$ and $\L(n)$ contain straight lines that form a basis of directions in each one of their tangent spaces, and we compute the geodesic equations in these manifolds. Finally, the quotient of $\L(n)$ by the diagonal action of the affine complex group and the re-enumeration of the vertices is described.}

\medskip

\noindent\textbf{2020 Math. Sub. Class.} \textit{Primary} 53A07, 53C22; \textit{Secondary} 53C15, 53C40.

\setlength{\parindent}{0cm}
\setlength{\parskip}{5pt}
\section{Introduction}
\label{int}

The spaces of polygons have been studied in many areas of mathematics, due to the implications and properties they hide. Many difficulties in different topics are reflected in these spaces. We mention some examples of this: in algebraic geometry, the difference between fine and coarse moduli spaces, and many phenomena from its generalization to algebraic stacks are reflected in spaces of triangles \cite{BEH}; in differential geometry, the existence or non-existence of global regularizing flows in simple or convex polygons \cite{ChowGli, CONN, SmBrouFra}; in topology, the compactification of moduli spaces of polygons \cite{HauKnu, KaMil1, KaMil2}. On the other hand, the study of geometric and topological properties have applications in areas like robotic motion, and pattern recognition \cite{Gott, Loz}.

Let $n\geq 3$. If the point $(z_1,z_2,\dots,z_n)\in\C^n$ is thought as the $n$-gon with consecutive vertices at $z_1,z_2,\dots,z_n\in\C$, then the set of $n$-gons in the plane with labeled vertices is equipped with the topology of $\C^n$. Here we study $n$-gons degenerated to a segment.

\begin{defi}
\label{n-segments}
A point in $\C^n$ is called an $n$-segment if all its vertices are collinear. $\L(n)\subset\C^n$ denotes the set of $n$-segments.
\end{defi}

An element in $\L(n)\subset\C^n$ defines a closed segment in $\C$ with $n$ labeled points on it. The $n$-segments appear in several investigations, for example:

\begin{itemize}
\item In \cite{AHT-MAN} it is proved that $\L(n)$ is contained in the boundary of the space of simple polygons, \textit{i.e.}, polygons without self-intersections. In particular, in the $n$-segments there is a bad asymptotic behavior that generates difficulties to study the boundary of the space of simple polygons \cite{AHT-JO}. The $n$-segments obtained as limits of convex polygons are called convex. In \cite{AHT-MAN} it is also shown that the convex $n$-segments are the only polygons that are limits of positively and negatively oriented convex polygons.

\item Let $\mathrm{r}=(r_1,\dots,r_n)\in\R^n$ be a vector with positive coordinates. Then $X(\mathrm{r})$ denotes the space of simple $n$-gons $(z_1,\dots,z_n)$ with sides lenght determined by $r$, \textit{i.e.}, $|z_{j+1}-z_j|=r_j$, for all $1\leq j\leq n$ (where $z_{n+1}=z_1$), modulo translations and rotations. Elements in $X(\mathrm{r})$ are called {\it linkages} and they are very studied in different areas \cite{BIEDL, CANTA, CONN, SHI-WOO}. A conjecture mentioned in \cite{CONN} says that the closure of $X(\mathrm{r})\subset\C^n$ is always contractible. The first counterexample is $X(6,4,2,4)$ and it was presented in \cite{SHI-WOO}. This is a counterexample since it contains a 4-segment. Actually, in \cite{AHT-MAN} it is proved that for $n=4$ all the counterexamples must contain a non-convex 4-segment.

\item Let $\mathrm{r}=(r_1,\dots,r_n)\in\R^n$ be a vector with positive coordinates such that $r_1+\cdots+r_n=1$. In \cite{KaMil1} Kapovich and Millson studied the topology of $M_\mathrm{r}$, the space of $n$-gons with sides lengths determined by $\mathrm{r}$, modulo translations and rotations (unlike $X(\mathrm{r})$, $M_\mathrm{r}$ includes the not simple $n$-gons). It is clear that in $M_\mathrm{r}$ there is an $n$-segment if and only if the coordinates of $\mathrm{r}$ satisfy$$\sum_{j=1}^n(-1)^{u_j}r_j=0\quad\text{for some } u_j\in\{0,1\}.$$The authors prove that $M_\mathrm{r}$ is a smooth manifold if and only if $\mathrm{r}$ does not satisfy the previous equation. Also, if $\mathrm{r}$ and $\mathrm{r}'$ are such that the closed segment between them does not contain a point that satisfies the previous equation, then the manifolds $M_\mathrm{r}$ and $M_{\mathrm{r}'}$ are diffeomorphic. That is, the topology changes precisely in the spaces $M_\mathrm{r}$ that contain $n$-segments. 

\item In \cite{KaMil2} Kapovich and Millson studied $n$-gons in $\R^3$ with side lengths determined by $\mathrm{r}=(r_1,\dots,r_n)$ as before. Two polygons are considered equivalent if they differ by an isometry of $\R^3$ and $\cM_\mathrm{r}$ denotes the corresponding moduli space. As in the previous item, $\cM_\mathrm{r}$ is a smooth manifold if and only if does not contain $n$-segments. In this case $\cM_\mathrm{r}$ is a complex analytic manifold with isolated singularities at the $n$-segments. Also, $\cM_\mathrm{r}$ is isomorphic to a weighted symplectic quotient of $(\S^2)^n$ with cusp points at the $n$-segments.
\end{itemize}

These examples show that the $n$-segments generate difficulties in the study of different features of spaces of polygons.

In this paper, we exhibit that $\L(n)\subset\C^n$ has interesting geometrical properties by itself. For example, if $V_n=\{(0,z_2,\dots,z_n)\in\C^n\}$ and $\M(n):=\L(n)\cap V_n$, then $\L(n)\subset\C^n$ is foliated by isometric copies of $\M(n)$ as follows
\[\L(n)=\bigcup_{b\in\C}\Big\{\M(n)+(b,b,\dots,b)\Big\}.\]
It turns out that $\M(n)$ and $\L(n)$ contain a big part of their tangent spaces, in particular, in each tangent space, $\M(n)$ contains a basis of $n$ perpendicular straight lines (Theorem \ref{Teorema_reglado}), and $\L(n)$ contains a basis of $n+2$ straight lines (Theorem \ref{Teorema_reglado2}). This contrasts with the case of surfaces in $\R^3$, where the only three-ruled surface is the plane \cite[Chap. 16]{FD-TS}. From these properties, it is striking that in $\M(n)$ and $\L(n)$ is complicated to find geodesics other than straight lines.

In addition to the space of $n$-segments, in this paper we are going to describe the space of shapes of $n$-segments without labels at the vertices that we define below.

\subsection{Shapes of $n$-segments and shapes without labelled vertices}\label{1A}

Let $\cA_{\C}:=\{f(z)=az+b\colon a\in\C^*:=\C\smallsetminus\{0\},\,b\in\C\}$ and $\cD_n:=\{(b,\dots,b)\in\C^n\}$ denote the Complex Affine Group and the diagonal in $\mathbb{C}^{n}$, respectively. Consider the action of $\cA_{\C}$ in $\C^n$ defined by
\[(az+b,(z_1,\dots,z_n))=(az+b,Z)\longmapsto aZ+\vec{b}=(az_1+b,\dots,az_n+b),\]
where $\vec{b}:=(b,\dots,b)\in\cD_n$. We say that two $n$-gons $Z$ and $W$ are equivalent if there exists $az+b\in \mathcal{A}_{\mathbb{C}}$ such that $(az+b,Z)\mapsto W$, \textit{i.e.} they have the same shape since they differ by a composition of translation, rotation and homothecy. Notice that two polygons in the subspace $V_n$ have the same shape if and only if they differ by a non-zero complex factor. The quotient$$P(n):=\left(\C^n\smallsetminus\cD_n\right)\big/\cA_{\C}\cong\P_{\C}V_n\cong\C\P^{n-2}$$is called the \textit{space of shapes of $n$-gons with labelled vertices}. The natural projection is denoted by $\eta\colon\C^n\smallsetminus\cD_n\to P(n)$. Notice that if $S^{2n-3}\subset V_n$ denotes the unitary sphere, then the restriction $\eta\colon S^{2n-3}\to P(n)$ is the Hopf fibration.

Two $n$-segments in $V_n \cap \mathbb{R}^n =\{(0,x_2,\dots,x_n)\colon x_j \in\R\}$ are equivalent if and only if they differ by a non-zero real factor. Then $\eta(\L(n))\subset P(n)$ is a copy of the projective space $\R\P^{n-2} \cong \S^{n-2}\big/\{X,-X\}$, with $\S^{n-2}\subset V_n \cap \mathbb{R}^n$ the unitary sphere. 

The linear isomorphism $M(z_1,z_2,\dots,z_n)=(z_2,\dots,z_n,z_1)$ permutes the vertices preserving the cyclic order. It is clear that two $n$-gons $Z,W\in\C^n$ are equivalent under the action of $\cA_{\C}$ if and only if $M(Z)$ and $M(W)$ are equivalent. Then, there is a diffeomorphism $\mu\colon P(n)\to P(n)$, such that $\mu\circ\eta=\eta\circ M$. The elements in the quotient $P(n)\big/\langle\mu\rangle$ are shapes of $n$-gons without labeled vertices.

\begin{defi}
\label{withoutlabels}
The quotient $\mathscr{L}(n):=\eta(\L(n))/\langle\mu\rangle$ is the space of shapes of $n$-segments without labelled vertices.
\end{defi}

The elements of $\mathscr{L}(n)$ can be thought as $n$-segments with vertices arranged in a certain pattern but with no labels on them.

\begin{exa}
\label{eje1}
The space of shapes of triangles $P(3)$ is diffeomorphic to $\C\P^1\cong\S^2$, and $\eta\big(\L(3)\big)=\big\{\eta\big((0,1,x)\big)\in P(3)\colon x\in\R\cup\{\infty\}\big\}$ is a copy of the circle $\S^1$. The space $\mathscr{L}(3)=\eta\big(\L(3)\big)/\langle\mu\rangle$ is obtained through the quotient$$\Big\{\eta\big((0,1,x)\big)\in P(3)\colon x\in[0,1]\Big\}\Big/\Big\{\eta\big((0,1,0)\big),\eta\big((0,1,1)\big)\Big\},$$and it is also diffeomorphic to the circle $\S^1$.
\end{exa}

The topics discussed in this article are structured as follows: In Section \ref{sec-2} we describe the topology of the smooth manifolds $\M(n)$ and $\L(n)$. Section \ref{sec-3} is devoted to show that $\M(n)$ and $\L(n)$ contain a big part of their tangent spaces in each point. This is a curious property that makes these manifolds more than ruled submanifolds. In Section \ref{sec-5} we calculate the equations that determine the geodesics in $\mathbb{L}(n)$ and $\mathbb{M}(n)$, it turns out that it is very difficult to compute geodesics other than straight lines. Finally, in Section \ref{sec-6} we prove that $\mathscr{L}(n)$ is a spherical orbifold and describe its singular locus.

\section{Topology of $\M(n)$ and $\L(n)$}
\label{sec-2}

Since every $n$-segment is determined by a point $X\in\R^n$, then $\L(n)$ is exactly the set $\eta^{-1}\big(\eta\big(\R^n\smallsetminus\{\vec{0}\}\big)\big)\subset\C^n$. On the other hand, the elements in $\M(n)=\L(n)\cap V_n$ are $n$-segments contained in a straight line through the origin in $\C$, and they are of the form $a X$, with $a\in\C^*$ and $X\in V_{n}\cap\mathbb{R}^{n}\smallsetminus\{\vec{0}$\}. We conclude that $\M(n)$ is diffeomorphic to $\rho^{-1}\big(\rho\big(\R^{n-1}\smallsetminus\{\vec{0}\}\big)\big)$ with $\rho\colon\C^{n-1}\smallsetminus\{\vec{0}\}\to\C\P^{n-2}$ the natural projection.

\begin{thm}
\label{variedad}
$\M(n)\subset\C^n$ is an $n$-manifold diffeomorphic to the mapping torus,$$\left(\big (\R^{n-1}\smallsetminus\{\vec{0}\}\big)\times[0,\pi]\right) \Big/\Big\{(X,0)\sim(-X,\pi)\Big\}.$$
\end{thm}

\begin{proof}
If $a=re^{i\theta}$ with $r\neq 0$ and $\theta\in[\pi,2\pi)$, then $a X=e^{i(\theta-\pi)}(-rX)$ for every $X\in \mathbb{R}^{n-1}\smallsetminus \{\vec{0}\}$ and therefore, the points in $\M(n)$ are uniquely determined by $e^{i\theta}X$ with $\theta\in[0,\pi)$ and $X\in V_n\cap\R^n$.

The function $\psi\colon\left(\big(\R^{n-1}\smallsetminus\{\vec{0}\}\big)\times[0,\pi]\right)\big/\{(X,0)\sim(-X,\pi)\}\to\M(n)$,
\[\psi\big((x_1,\dots,x_{n-1}),\theta\big)=\big(0,e^{i\theta}x_2,\dots,e^{i\theta}x_{n-1}\big),\]
is well defined since $\psi(X,0)=\psi(-X,\pi)$ and it is a continuous bijection. Consider the argument branch in $\C$ such that $-\frac{\pi}{2}<\arg(z)<\frac{3\pi}{2}$. Let $\theta\colon\C^*\to[0,\pi]$ be the function with $\theta(x)=0$ if $x$ is real and positive, $\theta(x)=\pi$ if $x$ is real and negative, and for all $z\in\C\smallsetminus\R$,
\[\theta(z)=\begin{cases}\arg(z) & \text{if }~ Im(z)>0. \\ \arg(-z) & \text{if }~ Im(z)<0 ~.\end{cases}\]
If $Z=(0,z_2,\dots,z_n)\in\M(n)$ and $z_k\neq 0$, then $\psi^{-1}(Z)=(e^{-i\theta(z_k)}(z_2,z_3,\dots,z_n),\theta(z_k))$ is the inverse function of $\psi$. Both functions $\psi$ and $\psi^{-1}$ are differentiable since they are differentiable in each coordinate.
\end{proof}

\begin{cor}
\label{producto}
$\L(n)\subset\C^n$ is an $(n+2)$-manifold diffeomorphic to $\M(n)\times\C$.
\end{cor}

\begin{proof}
The function $((0,z_2,\dots,z_n),b)\mapsto(b,z_2+b,\dots,z_n+b)$ is a continuous bijection from $\M(n)\times\C$ to $\L(n)$. The inverse, given by $(z_1,z_2,\dots,z_n)\mapsto((0,z_2-z_1,\dots,z_n-z_1),z_1)$, is also continuous. From Theorem \ref{variedad}, it follows that this function is a diffeomorphism and $\L(n)$ is a manifold of real dimension $n+2$.
\end{proof}

\begin{exa}
From Theorem \ref{variedad}, it follows that $\mathbb{M}(3)$ is diffeomorphic to the open solid 3-torus without the center $\big(\B^2\times\S^1\big)\smallsetminus\big(\{0\}\times\S^1\big)$, where $\B^2\subset\R^2$ denotes the unitary open ball. Both, $\M(3)\subset V_3$ and $\L(3)\subset\C^3$ have real co-dimension one.
\end{exa}

\section{Straight lines of $\M(n)$ and $\L(n)$}
\label{sec-3}

A classical definition is that a smooth submanifold of $\R^m$ is $l$-ruled if through each point of it, there are $l$ distinct straight lines contained in it. This kind of submanifolds is uncommon; for example, the only $3$-ruled surface in $\R^3$ is the plane. In this section we describe the straight lines contained in $\M(n)$ and $\L(n)$, and we prove that both manifolds contain a big part of their tangent spaces at each point.

\begin{defi}
\label{copies}
Let $Z=z(0,c_2,\dots,c_n)\in\M(n)$  with $z\in\C^*$ and $c_j\in\R$ and $\vec{b}\in\cD_n$.
\begin{enumerate}[a.]
    \item $\C_{Z+\vec{b}}^{\ast}:=\big\{\lambda Z+\vec{b}\in\L(n)\colon \lambda\in\C^*\big\}.$
    \item $\R^{n-1}_{Z+\vec{b}}:=\big\{z(0,c_2+x_2,\dots,c_n+x_n)+\vec{b}\in\L(n)\colon (x_{2},\dots,x_n)\in\R^{n-1}\smallsetminus\{-(c_{2},\dots,c_{n})\}\big\}$.
    \item $\cD_{Z+\vec{b}}:=\{Z+\vec{b}+\vec{\omega}\in \mathbb{L}(n):\vec{\omega}\in \mathcal{D}_{n}\}$.
\end{enumerate} 
\end{defi}

It is clear that $\C_{Z+\vec{b}}^{\ast}$ is a copy of $\C^*$, $\R^{n-1}_{Z+\vec{b}}$ is a copy of $\R^{n-1}\smallsetminus\{\vec{0}\}$, $\cD_{Z+\vec{b}}$ is a translation of the diagonal $\cD_n\cong\C$ to the point $Z+\vec{b}$, and the spaces $\C_Z^{\ast}:=\C^{\ast}_{Z+\vec{0}}$ and $\R^{n-1}_Z:=\R^{n-1}_{Z+\vec{0}}$ are contained in $\M(n)$.

\begin{rem}
\label{spaces}
For all $Z\in\M(n)$ we have that:

\noindent
$\bullet~$ For $b\in\C^*$, $\C_{Z+\vec{b}}^{\ast}\cap \mathbb{M}(n)=\R^{n-1}_{Z+\vec{b}}\cap \mathbb{M}(n)=\emptyset$.

\noindent
$\bullet~$ For $b\in\C$, $\cD_{Z+\vec{b}}\cap \mathbb{M}(n)=\{Z\}$, $ \cD_{Z+\vec{b}}\cap\mathbb{C}_{Z+\vec{b}}^{\ast}=\{Z+\vec{b}\}$, $\cD_{Z+\vec{b}}\cap\mathbb{R}_{Z+\vec{b}}^{n-1}=\{Z+\vec{b}\}$, and $\mathbb{C}_{Z+\vec{b}}^{\ast}\cap \mathbb{R}_{Z+\vec{b}}^{n-1}=\{rZ+\vec{b}:r\in\R^*:=\mathbb{R}\smallsetminus\{0\}\}\cong\mathbb{R}^*$.
\end{rem}

\begin{defi}
For $Z=(z_1,\dots,z_n),W=(w_1,\dots,w_n)\in\C^n$, we denote $\langle\!\langle Z,W\rangle\!\rangle=\sum_{j=1}^nz_j\pmb{\cdot}w_j$, with $z_j\pmb{\cdot}w_j= (x_j+i y_j)\pmb{\cdot}(\hat{x}_j+i\hat{y}_j)=x_j\hat{x}_j+y_j\hat{y}_j$ the usual inner product in the plane.
\end{defi}

\subsection{Tangent spaces of $\M(n)$ and $\L(n)$}

Remember that the tangent space of the submanifold $M\subset\R^n$ at $x\in M$, denoted by $T_xM$, is defined as the set of all vectors $\alpha'(0)$ such that $\alpha: (-\varepsilon,\varepsilon) \to M$, is a differentiable curve that satisfies that $\alpha(0)=x$. In Subsection \ref{Subsec_Straightlines} we make an abuse of notation using $T_Z\L(n)$ to refer to the affine space $Z+T_Z\L(n)$.

\subsubsection*{Tangent space of $\M(n)$}

From now on we consider the charts of $\M(n)$ given by $\varphi_k\colon\cU_k\to\R^n$, where$$\cU_k:=\{z(0,r_2,\dots,r_{k-1},1,r_{k+1},\dots,r_n)\in\M(n)\colon z=x+iy\in\C^*,r_m\in\R\}\quad\text{and}$$ $$\varphi_k(z(0,r_2,\dots,r_{k-1},1,r_{k+1},\dots,r_n))=(r_2,\dots,r_{k-1},x,y,r_{k+1},\dots,r_n).$$
We will do the calculations for $k=2$, other cases are analogous. Let fix $Z=\zeta(0,1,c_3,\dots,c_n)\in\cU_2$. Consider the following $n$ curves in $\M(n)$:
\begin{equation}
\label{curve}
\alpha_j(s):=\begin{cases}
(1+is)Z & \text{if}\quad j=1,\\
\zeta\big(0,1,c_3,\dots,c_{j-1},c_j+s,c_{j+1},\dots,c_n\big) & \text{if}\quad j\in \{2,3,\dots,n\}.\end{cases}
\end{equation}
For all $j\in \{ 1,\dots, n \}$, $\alpha_j$ is a differentiable curve such that $\alpha_j(0)=Z$ and therefore, an orthogonal real basis of the tangent space $T_Z\M(n)$ is given by
\begin{equation}
\label{alphas_primas}
\{\alpha'_1(0),\alpha'_2(0),\dots,\alpha'_{n}(0)\} = \{ iZ, \zeta e_2, \dots , \zeta e_n  \},
\end{equation}
where $e_j\in\C^n$ denotes the $j$-th vector in the complex canonical basis. Furthermore, $\C_Z^{\ast}\cup\{0\}$ and $\R_Z^{n-1}\cup\{\vec{0}\}$ are subspaces of $T_Z\M(n)$ and the sets$$\{ \alpha_1'(0),Z=\alpha_2'(0)+c_3\alpha_3'(0)+\cdots+c_n\alpha_n'(0)\}\quad\text{and}\quad\{\alpha_2'(0),\alpha_3'(0),\dots,\alpha_n'(0)\}$$are orthogonal basis of them, respectively.

\subsubsection*{Tangent space of $\L(n)$}
\label{susection_tangent}

In this case we will use the charts of $\L(n)$ given by $\hat{\varphi}_k\colon \hat{\cU}_l\to \R^{n+2}$, where$$\hat{\cU}_l=\{z(0,r_2,\dots,r_{k-1},1,r_{k+1},\dots,r_n)+\vec{b}\colon z=x+iy\in\C^*,r_m\in\R,b=u+iv\in\C\},$$and$$\hat{\varphi}_k\big(z(0,r_2,\dots,r_{k-1},1,r_{k+1},\dots,r_n)+\vec{b}\,\big)=(r_2,\dots,r_{k-1},x,y,r_{k+1},\dots,r_n,u,v).$$Again, we only work in the case $k=2$. Let fix $\hat{Z}=Z+\vec{b}\in\hat{\cU}_2$ with $Z=\zeta(0,1,c_3,\dots,c_n)\in\cU_2$ and $\vec{b}\in\cD_n$. For $j\in\{1,2,\dots,n\}$, the curve $\hat\alpha_j(t)=\alpha_j(t)+\vec{b}$ belongs to $\L(n)$, where $\alpha_j(t)$ are the functions in equation (\ref{curve}), $\hat{\alpha}_j(0)=\hat{Z}$, and $\hat{\alpha}'_j(0)=\alpha'_j(0)$. These $n$ vectors determine orthogonal directions in the tangent space $T_{\hat{Z}} \L(n)$. Consider $\hat{\alpha}_{n+1},\hat{\alpha}_{n+2}\colon(-\varepsilon,\varepsilon)\to\L(n)$ given by $\hat{\alpha}_{n+1}(t)=\hat{Z}+s\vec{1}$ and $\hat{\alpha}_{n+2}(t)=\hat{Z}+s\vec{i}$. These two curves satisfy that $\hat{\alpha}_{n+1}(0)=\hat{\alpha}_{n+2}(0)=\hat{Z}$, and the vectors $\hat{\alpha}'_{n+1}(0)=\vec{1},\hat{\alpha}'_{n+2}(0)=\vec{i}\in T_{\hat{Z}} \L(n)$ are orthogonal. We conclude that
\begin{equation}\label{alphaprimaL}
\{\hat{\alpha}'_1(0),\dots,\hat{\alpha}'_n(0),\hat{\alpha}'_{n+1}(0),\hat{\alpha}'_{n+2}(0)\}=\{ iZ, \zeta e_2, \dots , \zeta e_n,\vec{1},\vec{i}\}\end{equation}
is a non-orthogonal real basis of $T_{\hat{Z}} \L(n)$. Note that the sets $\{ \hat{\alpha}_1'(0),\hat{Z}=\hat{\alpha}_2'(0)+c_3\hat{\alpha}_3'(0)+\cdots+c_n\hat{\alpha}_n'(0)\}$, $\{\hat{\alpha}_2'(0),\hat{\alpha}_3'(0),\dots,\hat{\alpha}_n'(0)\}$ and $\{\hat\alpha'_{n+1}(0),\hat\alpha'_{n+2}(0)\}$ are orthogonal basis of $\C_{Z+\vec{b}}^{\ast}\cup\{0\}$, $\R_{Z+\vec{b}}^{n-1}\cup\{\vec{0}\}$ and $\mathcal{D}_{Z+\vec{b}}$, respectively.

\subsection{Straight lines of $\mathbb{M}(n)$ and $\mathbb{L}(n)$} \label{Subsec_Straightlines}

\subsubsection*{Straight lines of $\mathbb{M}(n)$}

\begin{prop}
\label{lines}
Let $Z,W\in\M(n)$. The segment $\sigma(t)=(1-t)Z+tW$, with $t\in[0,1]$, is contained in $\M(n)$ if and only if $W\in\big(\R^{n-1}_Z\cup\C_Z^{\ast}\big)\smallsetminus\{rZ\colon r\in(-\infty, 0]\}$.
\end{prop}

\begin{proof}
From Definition \ref{copies}, it follows that $\sigma\subset\M(n)$ for $W\in\big(\R^{n-1}_Z\cup\C_Z^{\ast}\big)\smallsetminus\{rZ\colon r\in(-\infty, 0]\}$. Suppose that $Z=(0,z_2,\dots,z_n)$ and $W=(0,w_2,\dots,w_n)\notin \R^{n-1}_Z\cup\C_Z^{\ast}$. There are two cases:

$I-$\textit{There is no $2\leq j\leq n$ with $z_j\neq 0$ and $w_j\neq 0$ simultaneously.} Since $W\notin\R^{n-1}_Z$, then for all $0<t<1$, the $n$-gon $\sigma(t)=(1-t)Z+tW$ has non-zero vertices at two different lines through $0\in\C$, then it is not an $n$-segment.

$II-$\textit{There is $2\leq j\leq n$ with $z_j\neq 0$ and $w_j\neq 0$ simultaneously.} Since $W\notin\C_Z^{\ast}$, then there is $m\neq j$, such that $z_m=rz_j$ and $w_m=sw_j$ with $r,s\in\R^*$. For $0<t<1$, the complex numbers $(1-t)z_j+tw_j$ and $(1-t)z_m+tw_m=(1-t)rz_j+tsw_j$ are not collinear with 0 ($z_j$ and $w_j$ are linearly independent over $\R$). We conclude that $\sigma(t)=(1-t)Z+tW$ is not an $n$-segment.
\end{proof}

\begin{thm}\label{Teorema_reglado}
For all $Z \in \M(n)$, there are $n=\dim\big(\M(n)\big)$ orthogonal lines contained in the intersection $\M(n)\cap T_Z\M(n)$. Furthermore, there are spaces isometric to $\R^{n-1}\smallsetminus \{ \vec{0}\}$ and $\R^{2}\smallsetminus \{\vec{0}\}$ that are contained in $\M(n)\cap T_Z\M(n)$. 
\end{thm}

\begin{proof}
Let $Z=\zeta(0,c_2,c_3,\dots,c_n)\in\M(n)$. The spaces generated by $$\{\alpha_2'(0),\alpha_3'(0),\dots,\alpha_n'(0)\}\quad\text{and}\quad\{\alpha_1'(0),  Z=c_2 \alpha_2'(0)+c_3\alpha_3'(0)+\cdots+c_n\alpha_n'(0)\}$$ are contained in $T_Z\M(n)$, where the derivatives $\alpha'_j$ are the functions in equation \eqref{alphas_primas}. Such spaces correspond to $\mathbb{R}^{n-1}_{Z}\cup\{\vec{0}\}$ and $\mathbb{C}^{\ast}_{Z}\cup\{\vec{0}\}$, respectively. It follows that $\mathbb{R}^{n-1}_{Z}$ $\cong \mathbb{R}^{n-1}\smallsetminus \{\vec{0}\} $ and $\mathbb{C}^{\ast}_{Z}\cong \mathbb{R}^2 \smallsetminus \{\vec{0}\}$ are contained in $\mathbb{M}(n)\cap T_{Z}\mathbb{M}(n)$. Taking $n-1$ orthogonal lines in $\mathbb{R}^{n-1}_{Z}$ through $Z$ that do not cross $\vec{0}$, and  the line $(1+it)Z$ contained in $\mathbb{C}^{\ast}_{Z}$, we get $n$ orthogonal lines in $\mathbb{M}(n)\cap T_{Z}\mathbb{M}(n)$.
\end{proof}

\subsubsection*{Straight lines of $\mathbb{L}(n)$}

\begin{prop}
Let $Z=(z_1,z_2,\dots,z_n), W=(w_1,w_2,\dots,w_n)\in \mathbb{L}(n)$. The segment $\sigma(t)=(1-t)Z+tW$, with $t\in [0,1]$, is contained in $\mathbb{L}(n)$ if and only if $W-\vec{w}_1\in\big(\R^{n-1}_{Z-\vec{z}_1}\cup\C_{Z-\vec{z}_1}^{\ast}\big)\smallsetminus\{r(Z-\vec{z}_1)\colon r\in\R,~r\leq 0\}$.
\end{prop}

\begin{proof}
If $\sigma(t)=(1-t)Z+tW$ and $\hat{\sigma}(t)=(1-t)(Z-\vec{z}_1)+t(W-\vec{w}_1)$, then the segment $\sigma(t)-\hat{\sigma}(t)=(1-t)\vec{z}_1+t\vec{w}_1$ with $t\in[0,1]$, is contained in the diagonal $\cD_n\subset\mathbb{C}^{n}$. So, $\sigma(t)\in\cD_n$ if and only if $\hat{\sigma}(t)\in\cD_n$, and therefore, $\sigma(t)\in\mathbb{L}(n)$ if and only if $\hat{\sigma}(t)\in\mathbb{M}(n)$. We conclude the proof using Proposition \ref{lines}.
\end{proof}

\noindent


\begin{thm}\label{Teorema_reglado2}
For all $\hat{Z}\in \mathbb{L}(n)$, there are $n+2=\dim(\mathbb{L}(n))$ linearly independent straight lines contained in the intersection $\mathbb{L}(n)\cap T_{\hat{Z}}\mathbb{L}(n)$. Furtheremore, there are spaces isometric to $\mathbb{R}^{n+1}\smallsetminus \mathbb{R}^{2}$ and $\mathbb{R}^{2}\smallsetminus\{\vec{0}\}$ contained in $\L(n)\cap T_{\hat{Z}}\L(n)$. 
\end{thm}

\begin{proof}
Let $\hat{Z}=Z+\vec{b} \in \mathbb{L}(n),$ with $Z=\zeta(0,c_2,c_3,\dots,c_n)\in\M(n)$ and $\vec{b}\in\cD_{n}$. The spaces generated by $$\{\alpha_2'(0),\alpha_3'(0),\dots,\alpha_{n+1}'(0),\alpha_{n+2}'(0)\}~~\text{and}~~\{\alpha_1'(0),Z=\alpha_2'(0)+c_3\alpha_3'(0)+\cdots+c_n\alpha_n'(0)\}$$ are contained in $T_{\hat{Z}}\L(n)$, where the $\alpha'_j$'s are the functions in equation \eqref{alphaprimaL}. Such spaces correspond to $R:=\langle\mathbb{R}^{n-1}_{\hat{Z}}\cup\{\vec{0}\},\mathcal{D}_{\hat{Z}}\cup\{\vec{0}\}\rangle$ and $\mathbb{C}^{\ast}_{\hat{Z}}\cup\{\vec{0}\}$, respectively. Notice that $R\smallsetminus \mathbb{L}(n)=\mathcal{D}_{\hat{Z}}\cong\mathbb{R}^{2}$. Then $R\smallsetminus\mathcal{D}_{\hat{Z}}\cong \mathbb{R}^{n+1}\smallsetminus\mathbb{R}^{2}$ and $\mathbb{C}_{\hat{Z}}^{\ast}\cong \mathbb{R}^{2}\smallsetminus\{0\}$ are contained in $\mathbb{L}(n)\cap T_{\hat{Z}}\mathbb{L}(n)$. Also, we can take $n+1$ orthogonal lines in $R\smallsetminus \mathcal{D}_{\hat{Z}}$ passing through $\hat{Z}$ and the line $(1+ti)Z+\vec{b}\subset\mathbb{C}_{\hat{Z}}^{\ast}$, obtaining $n+2$ lines through $\hat{Z}$ contained in $\mathbb{L}(n)$ that span a copy of $\mathbb{R}^{n+2}$.
\end{proof}

\section{Geodesic equations in $\M(n)$ and $\L(n)$}
\label{sec-5}

From the Remark \ref{spaces}, the geodesics on $\C_Z^{\ast}$ and $\R^{n-1}_Z$ are straight lines. On the other hand, $-Z\in\C_Z^{\ast}$ for all $Z\in\M(n)$, but the segment from $Z$ to $-Z$ is not contained in $\M(n)$ because it passes through $\vec{0}\in\cD_n$. We conclude that $\M(n)$ is not geodesically convex. Similarly happens with $\L(n)$. For the same reason,  $\M(n)$ and $\L(n)$ are not geodesically complete manifolds.

\subsection{Geodesic equations in $\M(n)$}
\label{EqM(n)}
  
Suppose that $\gamma\colon(-\varepsilon, \varepsilon)\to\cU_2$, $\gamma(t)=z(t)\big(0,1,r_3(t),\dots,r_{n}(t)\big)$ is a geodesic in $\M(n)$ with arbitrary initial conditions $\gamma(0)=Z$ and $\gamma'(0)\in T_Z\M(n)$.  

Since a geodesic $\gamma$ in a submanifold of $\C^n$ is characterized by having constant speed and by vanishing the orthogonal projection of $\gamma''(t)$  to the tangent space $T_{\gamma(t)}\M(n)$ for each $t$, our goal is to get equations equivalent to these conditions. We start with the last condition. This means that the acceleration vector \[\gamma''(t)=\big(0,z''(t),r_3''(t)z(t)+2r_3'(t)z'(t)+r_3(t)z''(t),\dots,r_n''(t)z(t)+2r_n'(t)z'(t)+r_n(t)z''(t)\big)\]
belongs to the orthogonal complement of the tangent space $T_{\gamma(t)}\M(n)$ in $\C^n$, \textit{i.e.} $\langle\!\langle\gamma''(t),V\rangle\!\rangle=0$ for any $V \in T_{\gamma(t)}\M(n)$. Then, the equations $\langle\!\langle\gamma''(t), v_j(t) \rangle\!\rangle=0$ with $v_j(t):=\alpha'_j(0)(t)$ and $\{v_1(t),v_2(t),\dots,v_n(t)\}\subset T_{\gamma(t)}\mathbb{M}(n)$ the orthogonal basis constructed in Subsection \ref{susection_tangent}, must to be satisfied. We proceed to compute these $n$ equations.

\begin{enumerate}[\bfseries I.]

\item $\langle\!\langle\gamma''(t),v_1(t) \rangle\!\rangle= \langle\!\langle\gamma''(t), i\gamma(t)\rangle\!\rangle=0$: We get $$\begin{array}{rcl}
0&=&z''(t)\pmb{\cdot}iz(t)+\sum_{j=3}^{n}\left([r_j''(t) z(t)+2r_j'(t)z'(t)+r_j(t) z''(t)]\pmb{\cdot}ir_j(t)z(t)\right)\\
&= & \dfrac{d}{dt}\left[\left(1+ \sum_{j=3}^n r_j(t)^2\right)[z'(t)\pmb{\cdot}iz(t)]\right]\,.\end{array}$$ Equivalently, there exists a constant $k_1\in\R$, such that
\begin{equation}
\label{GM1}
    \Big(1+\sum_{j=3}^n \big(r_j(t)\big)^2\Big)\big[z'(t)\pmb{\cdot}iz(t)\big]=k_1.
\end{equation}

\item $\langle\!\langle\gamma''(t),v_2(t)\rangle\!\rangle=\langle\!\langle\gamma''(t),z(t)e_{2}\rangle\!\rangle=0$: It is equivalent to
\begin{equation}
\label{GM2}
    z''(t)\pmb{\cdot}z(t)=0. 
\end{equation}
Then the curve $z\colon(-\varepsilon,\varepsilon)\to\C$ is orthogonal to its acceleration.

\item $\langle\!\langle\gamma''(t),v_j(t)\rangle\!\rangle=\langle\!\langle\gamma''(t),z(t)e_{j}\rangle\!\rangle=0$ for $3\le j\le n$: Then we have that 
\begin{equation}\begin{split}\label{GM3a}0=&r_j''(t)\lVert z(t)\rVert^2+2r_j'(t)(z'(t)\pmb{\cdot}z(t))+r_j(t)(z''(t)\pmb{\cdot}z(t))\\=&r_j''(t)\lVert z(t)\rVert^2+2r_j'(t)(z'(t)\pmb{\cdot}z(t)),\end{split}\end{equation}
where the last equality follows by applying the equation \eqref{GM2}. This is equivalent to
\begin{equation}
\label{GM3b}
r_j'(t)\,\lVert z(t)\rVert^2=k_j\,,
\end{equation}
where $k_j\in\R$ is constant. From these equations it follows that the functions $r_j\colon(-\varepsilon,\varepsilon)\to\R$ satisfy:

\begin{enumerate}[a)]
\item If $r_j'(t_0)=0$ for some $t_0\in(-\varepsilon, \varepsilon)$, then $r_j(t)$ is a constant function since $z(t)\neq 0$.

\item If $r_j'(t_0)\neq 0$ for some $t_0$, then $r_j(t)$ is a monotone function.

\item If $r_j'(t)\neq 0$, then for every $m\neq j$ such that $r_m'(t)\neq 0$, there is a constant $K_m\in\R$, such that $r_m(t)=\frac{k_m}{k_j}r_j(t)+K_m$.
\end{enumerate}

\item Since the speed $\langle\!\langle\gamma'(t), \gamma'(t) \rangle\!\rangle=\lVert \gamma'(t)\rVert^2=k_0$ is constant, we have that \begin{equation}\begin{split}\label{GM4}k_0=\sum_{j=3}^n(r_j'(t))^2\lVert z(t)\rVert^2+\sum_{j=3}^n 2r_j'(t)r_j(t)z(t)\pmb{\cdot}z'(t)+\left(1+\sum_{j=3}^nr_j^2(t)\right)\lVert z'(t)\rVert^2.\end{split}\end{equation} 
Using equation \eqref{GM3b} and the procedure to obtain it, we get
\begin{equation}
\label{Eq4}
\dfrac{k_0-\sum_{j=3}^n k_j\Big(\dfrac{r_j(t)}{r'_j(t)} \Big)'r'_j(t)}{1+\sum_{j=3}^n r_j(t)^2}=\lVert z'(t)\rVert^2 .
\end{equation}

\end{enumerate}

It is easy to check that straight lines satisfy these conditions. Unfortunately, we could not find other solutions than straight lines. This suggests that the manifold $\M(n)$ is very ``crooked'' inside the complex subspace $V_n=\{(0,z_2,\dots,z_n)\in\C^n\}$. Even in the case of 3-segments, we fail to find a geodesic in $\M(3)$ different than lines. However, it is possible to know nontrivial geodesics in $\M(n)$ from a nontrivial geodesic in $\M(3)$.

\begin{thm}
Let $\beta(t)=z(t)\big(0,1,r(t)\big)\subset\cU_2$ be a geodesic of $\M(3)$ that is not a straight line, and let  $$\gamma(t)=z(t)\big(0,1,a_{3}s_3(t),a_{4}s_4(t),\dots,a_{n}s_n(t)\big)\subset \mathbb{M}(n),$$ be a curve, where $a_{j}\in \mathbb{R}$ is constant and $s_{j}(t)\equiv 1$ or $s_{j}(t)\equiv r(t)$. Then $\gamma(t)$ is a geodesic in $\M(n)$ if and only if $$1+\sum_{j\in A_{1}}a_{j}^{2}=\sum_{j\in A_{2}}a_{j}^{2},$$ where $A_{1}=\{3\leq j\leq n\colon s_{j}(t)\equiv 1\}$ and $A_{2}=\{3\leq j\leq n\colon s_{j}(t)\equiv r(t)\}$.
\end{thm}

\begin{proof}
   The geodesic $\beta(t)=\big(0,z(t),r(t)z(t)\big)\subset\cU_2$ satisfies the equations
\smallskip
$$0=z''(t) \pmb{\cdot}z(t),\qquad k=r'(t)\lVert z(t)\rVert^2,\qquad k_1=\left[z'(t)\pmb{\cdot} iz(t)\right]\left(1+r(t)^2\right)$$
$$k_0=\left(1+r^2(t)\right)\lVert z'(t)\rVert^2+(r'(t))^2\lVert z(t)\rVert^2+2r'(t)r(t)\left(z(t)\pmb{\cdot} z'(t)\right),$$
where $k,k_1,k_0\in\R$ and $k_0>0$.

Notice that \[\frac{d}{dt}[z'(t)\pmb{\cdot} iz(t)]=z''(t)\pmb{\cdot} iz(t).\] Since $z(t)$ is not a straight line we know that $z''(t)$ is not identically $0$. From this and $z''(t)\pmb{\cdot} z(t)=0$, we conclude $z''(t)\pmb{\cdot} iz(t)$ is not identically $0$ and therefore $z'(t)\pmb{\cdot} iz(t)$ is not constant, and then $k_{1}\neq 0$.

$\Rightarrow]$ Let $A:=1+\sum_{j\in A_{1}}a_{j}^{2}$ and $A':=\sum_{j\in A_{2}}a_{j}^{2}$. From condition \textbf{I} for $\gamma(t)$ and conditions on $\beta(t)$ we have 
\begin{equation}\label{k1k1}\frac{k_1}{1+\left(r(t)\right)^2}=z'(t)\pmb{\cdot} iz(t)=\frac{\kappa_1}{A+A'r(t)^2}\end{equation}
for some $\kappa_{1}\in \mathbb{R}$. Since $r(t)$ is not constant (in other case, $\beta(t)$ would be a straight line), it follows that $A=A'$.

$\Leftarrow]$ It is enough to show that the curve $\gamma(t)$ satisfies the conditions \textbf{I}, \textbf{II}, \textbf{III} and \textbf{IV}.

\textit{Condition \textbf{I}:} It follows from equation \eqref{k1k1} and $A=1+\sum_{j\in A_{1}}a_{j}^{2}=\sum_{j\in A_{2}}a_{j}^{2}=A'$, where $\kappa_{1}=Ak_{1}$.

\textit{Condition \textbf{II}:} The equation of this condition is identical for $\beta(t)$ and $\gamma(t)$.

\textit{Condition \textbf{III}:} We consider first the case when $s_{j}(t)\equiv 1$. As $r_{j}'(t)\equiv 0$ we conclude that $r_{j}'(t)\lVert z(t)\rVert^{2}$ is the constant function $0$.

Now, we consider the case when $s_{j}(t)\equiv r(t)$. In this case the condition is also satisfied since $a_{j}s_{j}'(t)\lVert z(t)\rVert^{2}=a_{j}k$ is constant.

\textit{Condition \textbf{IV}:} Notice that
\begin{align*}\lVert \gamma'(t)\rVert^{2}&=\lVert z'(t)\rVert^2+\sum_{j=3}^{n}\lVert (a_{j}s_{j}(t)z(t))'\rVert^2\\
&=\lVert z'(t)\rVert^2+\sum_{j\in A_{1}}\lVert (a_{j}z(t))'\rVert^2+\sum_{j\in A_{2}}\lVert (a_{j}r(t)z(t))'\rVert^2\\
&=\left(1+\sum_{j\in A_{1}}a_{j}^{2}\right)\lVert z'(t)\rVert^2+\sum_{j\in A_{2}}a_{j}^2\lVert (r(t)z(t))'\rVert^2\\ 
&=A\left(\lVert z'(t)\rVert^{2}+r^{2}(t)\lVert z'(t)\rVert^{2}+2r'(t)r(t)(z(t)\pmb{\cdot} z'(t))+(r'(t))^{2}\right)=Ak_{0},\end{align*}
which is constant.
\end{proof}

\subsection{Geodesic equations in $\L(n)$}

Now, let $\hat{\gamma}(t)=\gamma(t)+\vec{\lambda}(t)$ be a geodesic in $\hat{\cU}_2 \subset \L(n)$, with $\gamma(t)\subset\cU_2\subset\M(n)$, and let $\vec{\lambda}(t)$ be a curve contained in $\mathcal{D}_{n}$. If we denote $v_j(t):=\hat{\alpha}'_j(0)(t)$ for all $t\in(-\varepsilon,\varepsilon)$, then $\{v_1(t),v_2(t),\dots,v_{n+2}(t)\}$ is a basis of $T_{\hat{\gamma}(t)}\mathbb{L}(n)$. The geodesic conditions (analogous to the previous section) generate the following equations:

\begin{enumerate}[\bfseries I'.]
\item From $\langle\!\langle\hat{\gamma}''(t),v_1(t)\rangle\!\rangle=\langle\!\langle {\gamma}''(t)+\vec{\lambda}''(t), i \gamma(t)  \rangle\!\rangle=0$, we get
\begin{equation}
\label{GL1}
0=\Big[\Big(1+\sum_{j=3}^{n} r_j^2(t)\Big)z''(t)+ 2\Big(\sum_{j=3}^{n}r_j(t)r'
_j(t) \Big)z'(t)+\Big(1+\sum_{j=3}^n r_j(t)\Big)\lambda''(t)\Big]\pmb{\cdot} iz(t)     
\end{equation}

\item $\langle\!\langle\hat{\gamma}''(t),v_2(t)\rangle\!\rangle=\langle\!\langle {\gamma}''(t)+\vec{\lambda}''(t), z(t) e_2\rangle\!\rangle=0$: It is equivalent to
\begin{equation}
\label{GL2}
    z''(t)\pmb{\cdot} z(t) + \lambda''(t) \pmb{\cdot} z(t)=0. 
\end{equation}

\item $\langle\!\langle\hat{\gamma}''(t),v_j(t)\rangle\!\rangle=\langle\!\langle {\gamma}''(t)+\vec{\lambda}''(t), z(t) e_j\rangle\!\rangle=0$ for ${3\leq j\leq n:}$ From here and equation \eqref{GL1} we have that 
\begin{equation}
\label{GL3}
0=z(t) \pmb{\cdot} [r_j''(t) z(t) + 2 r_j'(t)z'(t)+(r_j(t)-1)z''(t)]
\end{equation}

\item $\langle\!\langle\hat{\gamma}''(t),v_{n+1}(t)\rangle\!\rangle=\langle\!\langle\hat{\gamma}''(t),v_{n+2}(t)\rangle\!\rangle=0$: From here we obtain that
\begin{equation}
0=\Big(1+\sum_{j=3}^{n} r_j(t)\Big)z''(t)+ 2\Big(\sum_{j=3}^{n} r'_j(t)\Big)z'(t)+\Big(\sum_{j=3}^{n}r''_j(t)\Big)z(t)+n\lambda''(t) 
\end{equation}

\item $\langle\!\langle\hat\gamma'(t), \hat\gamma'(t) \rangle\!\rangle=\lVert \gamma'(t)+\vec{\lambda}'(t)\rVert^2=k_0$ is constant: We have that

\begin{equation}\begin{split}
    \sum_{j=3}^n(r_j'(t))^2\lVert z(t)\rVert^2+\sum_{j=3}^n 2r_j'(t)r_j(t)z(t)\pmb{\cdot} z'(t)+\left(1+\sum_{j=3}^nr_j^2(t)\right)\lVert z'(t)\rVert^2\\
    +\left(\sum_{j=3}^{n}r_{j}'(t)\right)z(t)\pmb{\cdot} \lambda'(t)+\left(1+\sum_{j=3}^{n}r_{j}(t)\right)z'(t)\pmb{\cdot} \lambda'(t)+n\lVert\lambda'(t)\rVert^2=k_{0},\end{split}\end{equation} which is constant.
\end{enumerate}

\begin{prop}
    Let $\gamma(t)$ be a geodesic in $\mathbb{M}(n)$. Then $\hat{\gamma}(t)=\gamma(t)+\vec{\lambda}(t)$ is a geodesic in $\mathbb{L}(n)$ (where $\vec{\lambda}(t)\in \mathcal{D}_{n}$) if and only if $\lambda''(t)\equiv 0$, $\langle\!\langle \gamma''(t),\vec{1}\rangle\!\rangle\equiv 0$ and $\langle\!\langle \gamma''(t),\vec{i}\rangle\!\rangle\equiv 0$.
\end{prop}

\begin{proof} $\Rightarrow]$ Condition \textbf{I} on $\gamma$ and Condition \textbf{I'} on $\hat{\gamma}$ imply that 
$$\left(1+\sum_{j=3}^{n}r_{j}(t)\right)\lambda''(t)\pmb{\cdot} iz(t)=0.$$ 
We consider two cases:

\textit{Case $1+\sum_{j=3}^{n}r_{j}(t)=0$}: In this case it is fulfilled that the derivative of $1+\sum_{j=3}^{n}r_{j}(t)$ is identically 0. We use this in Condition \textbf{IV'} to obtain that $\lambda''(t)\equiv 0$.

\textit{Case $\lambda''(t)\pmb{\cdot} iz(t)=0$}: From conditions \textbf{II} and \textbf{II'} we deduce that $\lambda''(t)\pmb{\cdot} z(t)=0$. Since $z(t)\not =0$, we conclude that $\lambda''(t)\equiv 0$.

Then, in both cases, we get that $\lambda''(t)\equiv 0$. From last equation and Condition \textbf{IV'}, we get that $\langle\!\langle \gamma''(t),v_{n+1}(t)\rangle\!\rangle=\langle\!\langle \gamma''(t),v_{n+2}(t)\rangle\!\rangle\equiv 0$. Since $v_{n+1}(t)\perp v_{n+2}(t)$ we conclude that $\langle\!\langle \gamma''(t),\vec{1}\rangle\!\rangle\equiv 0$ and $\langle\!\langle \gamma''(t),\vec{i}\rangle\!\rangle\equiv 0$.

$\Leftarrow ]$ Notice that conditions \textbf{I'}, \textbf{II'} and \textbf{III'} are implied by conditions \textbf{I}, \textbf{II} and \textbf{III} and from the fact that $\lambda''(t)\equiv 0$. Also, condition \textbf{IV'} follows from $\lambda''(0)\equiv 0$, $\langle\!\langle \gamma''(t),\vec{1}\rangle\!\rangle\equiv 0$ and $\langle\!\langle \gamma''(t),\vec{i}\rangle\!\rangle\equiv 0$.

Condition \textbf{V'} is equivalent to prove that $\lVert \gamma'(t) \rVert^2+\lVert \vec{\lambda}'(t) \rVert^2+2\langle\!\langle \gamma'(t),\vec{\lambda}'(t)\rangle\!\rangle$ is a constant. Since $\gamma$ is a geodesic and $\lambda''\equiv 0$, then $\lVert\gamma'(t)\rVert$ and $\lVert\vec\lambda'(t)\rVert$ are constant functions. Also, since $\lambda'(t)\equiv a$ for some $a\in \C$, we have that $\langle\!\langle\gamma'(t),\hat\lambda'(t)\rangle\!\rangle=\langle\!\langle \gamma'(t),\vec{a}\rangle\!\rangle$. Finally, condition \textbf{V'} is satisfied since $$\frac{d}{dt}\langle\!\langle \gamma'(t),\vec{a}\rangle\!\rangle=\langle\!\langle \gamma''(t),\vec{a}\rangle\!\rangle+\langle\!\langle \gamma'(t),\frac{d}{dt}\vec{a}\rangle\!\rangle=\langle\!\langle \gamma''(t),\vec{a}\rangle\!\rangle=0.$$
\end{proof}

\section{The orbifold $\mathscr{L}(n)$}
\label{sec-6}

In this section, we prove that quotient $\mathscr{L}(n)$ (defined in the Introduction) is a spherical orbifold and describe its singular locus, which consists of the singular points. Moreover, when $n$ is a prime number, we prove that $\mathscr{L}(n)$ is a lens space.

\begin{rem}
The ends of an $n$-segment are the vertices at the boundary of the segment. The $n$-segments have at least two ends, for example, the $5$-segment $(-3-3i,0,2+2i,1+i,2+2i)$ has three ends. If $z_m$ and $z_M$ are different ends of $Z$, then $(Z-\vec{z}_m)/(z_M-z_m)$ and $(Z-\vec{z}_M)/(z_m-z_M)$ are the representatives of $Z$ with ends at $0$ and $1$. Since $\eta(\mathbb{L}(n))\cong \mathbb{RP}^{n-2}\cong \mathbb{S}^{n-2}/\{X,-X\}$ and the space of $n$-segments with ends at $0$ and $1$ is a double cover of $\eta(\mathbb{L}(n))$, we conclude this space is homeomorphic to $\mathbb{S}^{n-2}$.\end{rem}

A lifting of the diffeomorphism $\mu\colon\eta(\L(n))\to\eta(\L(n))$ to $V_n \cap \mathbb{R}^n$ is the function $\hat{\mu}(0,x_2,x_3,\dots,x_n)=(0,x_3-x_2,\dots,x_n-x_2,-x_2)$ (\textit{i.e.} $\mu\circ p=p\circ\hat{\mu}$). From now on, we identify $V_{n}\cap \mathbb{R}^{n}$ with $\R^{n-1}$ through the projection over the last $n-1$ coordinates, and $\cM$ denote the $(n-1)\times (n-1)$ linear matrix associated to $\hat{\mu}$.

\begin{prop}
\label{eigen}
The eigenvalues of $\cM$ are $e^{\frac{2\pi i}{n}}, e^{2\frac{2\pi i}{n}}, \dots, e^{(n-1)\frac{2\pi i}{n}}$, and the eigenvectors are$$B_k=(e^{k\frac{2\pi i}{n}}-1,e^{2k\frac{2\pi i}{n}}-1,\dots,e^{(n-1)k\frac{2\pi i}{n}}-1),~\text{for}~k\in\{1,2,\dots,n-1\}.$$
\end{prop}

\begin{proof}
Using induction and proceeding by minors in the last row, it can be shown that the characteristic polynomial of $\cM$ is the cyclothomic polynomial $x^{n-1}+\cdots+x+1$. Then the eigenvalues are the $n$-roots of the unity different from 1. It is straightforward to check that $\cM(B_k)=e^{k\frac{2\pi i}{n}}B_k$, for all $k\in\{1,2,\dots,n-1\}$.
\end{proof}

\begin{defi}
\label{matrix}
Let $R_{\theta}\in GL_2(\R)$ denote the rotation matrix of angle $\theta\in[0,2\pi]$ in the counter-clockwise. We denote $\cR_n\in GL_{n-1}(\R)$ as the following matrix:
\begin{enumerate}[a.]
\item for odd $n$, $\cR_n$ has $2\times2$ size blocks in the diagonal given by $R_{1\cdot\frac{2\pi}{n}},R_{2\cdot\frac{2\pi}{n}},\dots,R_{\frac{n-1}{2}\cdot\frac{2\pi}{n}}$, and zeros in the remaining entries.
\item for even $n$, $\cR_n$ has blocks in the diagonal given by $R_{1\cdot\frac{2\pi}{n}},R_{2\cdot\frac{2\pi}{n}},\dots,R_{\frac{n-2}{2}\cdot\frac{2\pi}{n}}$, $-1$, and zeros in the remaining entries.
\end{enumerate}
\end{defi}

By Proposition \ref{eigen}, the isomorphisms $\cM$ and $\cR_n$ are conjugated, actually, $\cR_n=\cB_n^{-1}\cM\cB_n$, where $\cB_n\in GL_{n-1}(\R)$ is the matrix with ordered columns $C_1,S_1,\dots,$ $C_{\frac{n-1}{2}},S_{\frac{n-1}{2}}$, given by 
\begin{equation}\label{eq1}
    C_{j}=\begin{pmatrix}
        \cos\frac{2\pi j}{n}-1 \\
        \cos\frac{2\cdot 2\pi j}{n}-1 \\
        \vdots \\
        \cos\frac{(n-1)2\pi j}{n}-1 \\
    \end{pmatrix}\qquad\qquad
    S_{j}=\begin{pmatrix}
        -\sin\frac{2\pi j}{n}-1 \\
        -\sin\frac{2\cdot 2\pi j}{n}-1 \\
        \vdots \\
        -\sin\frac{(n-1)2\pi j}{n}-1 \\
    \end{pmatrix}
\end{equation}
for odd $n$, and $C_{1}$, $S_{1}$, $\dots$, $C_{\frac{n-2}{2}}, S_{\frac{n-2}{2}}, -C_{\frac{n}{2}}$ for even $n$. 

We denote $\langle g,h\rangle$ the group generated by the elements $g$ and $h$, and $\nu(X)=-X$.

\begin{thm}
\label{quotient1}
The space $\mathscr{L}(n)$ is homeomorphic to the sperical orbifold $\S^{n-2}\big/\langle\cR_n,\nu\rangle$.
\end{thm}

\begin{proof}
Since $\cM$ is the lifting of $\mu$ to $\R^{n-1}$ and it is conjugated to $\cR_n$, then the quotient $\mathscr{L}(n)=\eta(\L(n))\big/\langle\mu\rangle=\left(\S^{n-2}\big/\langle\nu\rangle\right)\big/\langle\mu\rangle$ is homeomorphic to $\S^{n-2}\big/\langle\cR_n,\nu\rangle$. It is clear that $\langle\cR_n,\nu\rangle$ is a subgroup of the orthogonal group $O(n-1)$ (Definition \ref{matrix}).
\end{proof}

\begin{defi}
    Let $q$ be a positive integer. If $p_1,p_2,\dots,p_m$ are integers such that $gcd(p_1,p_2,\dots,p_m,q)=1$, and $G\colon\C^m\to\C^m$ is given by$$G(z_1,z_2,\dots,z_m)=\big(e^{2\pi i\frac{p_1}{q}}z_1,e^{2\pi i\frac{p_2}{q}}z_2,\dots,e^{2\pi i\frac{p_m}{q}}z_m\big),$$then the quotient $L_{q}(p_1,p_2,\dots,p_m):=\S^{2m-1}/\langle G\rangle$ is called an orbifold lens space. Also, if for all $t$ it is satisfied that $gcd(p_t,q)=1$, then $\langle G\rangle$ acts freely and $L_{q}(p_1,p_2,\dots,p_m)$ is a manifold, which is named a lens space.
\end{defi}

\begin{prop}
\label{impar}
Let $n$ be odd. Then $\mathscr{L}(n)$ is homeomorphic to $L_{2n}(n+2,n+4,\dots,2n-1)$ and therefore, $\mathscr{L}(n)$ is a manifold for $n$ prime.
\end{prop}

\begin{proof}Since $n$ and 2 are co-prime, then the group $\langle\cR_n,\nu\rangle$ is isomorphic to the cyclic group $\langle\nu\cR_n\rangle$. The identification $(x_1,x_2,\dots,x_{n-2},x_{n-1})\mapsto(x_1+ix_2,\dots,x_{n-2}+ix_{n-1})$ between $\R^{n-1}$ and $\C^{\frac{n-1}{2}}$, takes the generator $\nu\cR_n$ to the function
$$(z_1,z_2,\dots,z_{\frac{n-1}{2}})\mapsto\big(e^{2\pi i(\frac{n+2}{2n})}z_1,e^{2\pi i(\frac{n+4}{2n})}z_2,\dots,e^{2\pi i(\frac{2n-1}{2n})}z_{\frac{2n-1}{2}}\big).$$
Then $\mathscr{L}(n)=\S^{n-2}\big/\langle\cR_n,\nu\rangle=\S^{n-2}\big/\langle\nu\cR_n\rangle$ is homeomorphic to $L_{2n}(n+2,n+4,\dots,n+(n-1))$ (see \cite{NBEH}). Notice that for prime numbers $n\geq3$, $2n$ is co-prime with $n+2,n+4,\dots,2n-1$ and therefore, the quotient $L_{2n}(n+2,n+4,\dots,2n-1)$ is a lens space.
\end{proof}

\begin{exa}
\label{eje3}
The first three spaces of $n$-segments without labeled vertices are: From Example \ref{eje1}, $\mathscr{L}(3)\cong\S^1$. It turns out that $\mathscr{L}(4)\cong\S^2\big/\langle\cR_4,\nu\rangle$ is homeomorphic to the disc $\D^2=\{z\in\C\colon|z|\leq1\}$. By Proposition \ref{impar}, the space $\mathscr{L}(5)$ is diffeomorphic to the three-dimensional lens space $L_{10}(7,9)$.
\end{exa}

\subsection{Singular locus in $\mathscr{L}(n)$}

As we mentioned before, the singular locus in $\mathscr{L}(n)$ is the set of its singular points. When $n$ is a prime number, $\mathscr{L}(n)$ is a manifold and therefore its singular locus is empty, then we suppose that $n$ is a composite number. 

Since $$\mathscr{L}(n)  \cong \eta(\mathbb{L}(n))/\langle \mu\rangle \cong \mathbb{S}^{n-2}/\langle \mathcal{R}_{n},\nu\rangle$$ (see Theorem \ref{quotient1} and Definition \ref{withoutlabels}), the set of these singular points coincide with the quotient by $\langle \mu \rangle$ of the subset of $n$-segments in $\eta(\L(n))$ which are fixed by a non trivial element of $\langle \mu \rangle$, or equivalently, with the quotient by $\langle\cR_n,\nu\rangle$ of the set of points in $\S^{n-2}$ which are fixed by a non-trivial element of $\langle\cR_n,\nu\rangle$.

Let $X\in\eta(\L(n))$ and $j\in \mathbb{N}$. If $gcd(j,n)=g$, then $\mu^j(X)=X$ if and only if $\mu^g(X)=X$. Therefore it is only necessary to study the cases in which $j$ is a divisor of $n$. From now on we suppose that $n=jk$ with $0<j<n$. For the following, let $\ell^{j}(n):=\{X\in\S^{n-2}\colon\cR_n^j(X)=X\text{ or }\cR_n^j(X)=-X\}$.

\begin{thm}
\label{esferas} Let $j$ be a proper divisor of $n$, and $k=\frac{n}{j}$. Then 
 $\ell^{j}(n)$ is homemorphic to $\S^{j-2}$ for odd $k$ and to $\S^{j-2}\cup\S^{j-1}$ for even $k$.
\end{thm}

\begin{proof}
Supposing $n=jk$ as before, $\cR_n^j$ is conformed by the blocks $R_{1\cdot\frac{2\pi}{k}},R_{2\cdot\frac{2\pi}{k}},\dots,R_{\frac{n-1}{2}\cdot\frac{2\pi}{k}}$ for odd $n$ (see Definition \ref{matrix}); and by the blocks $R_{1\cdot\frac{2\pi}{k}}, R_{2\cdot\frac{2\pi}{k}},\dots,R_{\frac{n-2}{2}\cdot\frac{2\pi}{k}}$ and the entry $(-1)^{j}$ for even $n$. 

We study $\ell^{j}(n)$ as $\ell_-^{\,j}(n) \cup \ell_+^{\,j}(n)$ where $\ell_-^{\, j}(n):=\{X\in\S^{n-2}\colon \cR_n^j(X)=-X\}$ and $\ell_+^{\,j}(n):=\{X\in\S^{n-2}\colon \cR_n^j(X)=X\}$. 

 Thus, we look for when $R_{m \cdot\frac{2\pi}{k}}$ is the identity $2\times 2$ block in order to obtain the subspace of $\mathbb{R}^{n-1}$ such that $\mathcal{R}^{j}_{n}(X)=X$ and therefore the respective sphere $\ell_+^{\,j}(n)$; and we look for when $R_{m \cdot\frac{2\pi}{k}}$  is the negative of the identity $2\times 2$ blocks to obtain the subspace of $\mathbb{R}^{n-1}$ such that $\mathcal{R}^{j}_{n}(X)=-X$ and therefore the respective sphere $\ell_-^{\,j}(n)$. Also, for even $n$ we have to check the sign of $(-1)^{j}$.

If $n$ is odd, then $j$ and $k$ are odd. Hence, the blocks corresponding to $R_{m \cdot \frac{2 \pi}{k}}$ where $m\in \{k,2k\dots, \frac{(j-1)k}{2}\}$ conform the $2\times 2$ identity blocks in the diagonal of $\mathcal{R}^{j}_{n}$, and there are no $2\times 2$ blocks in the diagonal equal to the negative of the identity. In this case we conclude that $\ell_{+}^{\,j}(n)\cong \mathbb{S}^{j-2}$ and $\ell_{-}^{\,j}(n)=\emptyset$.

If $n$ is even and $k$ is odd, then $j$ is even. In this case, the blocks corresponding to $R_{m \cdot \frac{2 \pi}{k}}$ where $m\in \{k,2k\dots, \frac{(j-2)k}{2}\}$ conform the $2\times 2$ identity blocks in the diagonal of $\mathcal{R}^{j}_{n}$ and also the entry $(-1)^{j}$ is a $1$ in the diagonal of such matrix. Moreover, there are no $2\times 2$ blocks in the diagonal equal to the negative of the identity. For this case we also conclude that $\ell_{+}^{\,j}(n)\cong \mathbb{S}^{j-2}$ and $\ell_{-}^{\,j}(n)=\emptyset$.

If $n$ and $k$ are even and $j$ is odd, the blocks corresponding to $R_{m \cdot \frac{2 \pi}{k}}$ where $m\in \{k,2k\dots, \frac{(j-1)k}{2}\}$ conform the $2\times 2$ identity blocks in the diagonal of $\mathcal{R}^{j}_{n}$; and the $2\times 2$ blocks corresponding to $R_{m}\cdot \frac{2\pi}{k}$ where $m\in\{\frac{k}{2},\frac{3k}{2},\dots, \frac{(j-2)k}{2}\}$ conform the $2\times 2$ blocks that are equal to the negative identity, also the $(-1)^{j}$ entry is a $-1$ in the diagonal of such matrix. In this case we have that $\ell_{+}^{\,j}(n)\cong \mathbb{S}^{j-2}$ and $\ell_{-}^{\,j}(n)\cong\mathbb{S}^{j-1}$. 

If $n$, $k$ and $j$ are even, the blocks corresponding to $R_{m \cdot \frac{2 \pi}{k}}$ where $m\in \{k,2k\dots, \frac{(j-2)k}{2}\}$ conform the $2\times 2$ identity blocks in the diagonal of $\mathcal{R}^{j}_{n}$, and also the entry $(-1)^{j}$ is a $1$ in the diagonal of such matrix; and the $2\times 2$ blocks corresponding to $R_{m}\cdot \frac{2\pi}{k}$ where $m\in\{\frac{k}{2},\frac{3k}{2},\dots, \frac{(j-1)k}{2}\}$ conform the $2\times 2$ blocks that are equal to the negative identity in the diagonal of $\mathcal{R}_{n}^{j}$. In this case we have that $\ell_{+}^{\,j}(n)\cong \mathbb{S}^{j-2}$ and $\ell_{-}^{\,j}(n)\cong\mathbb{S}^{j-1}$.
\end{proof}

Given a set $A$, we denote the closed cone over $A$ by $\cC_A:=A\times[0,1]\big/\big(A\times\{1\}\big)$.

\begin{thm}
\label{locusing}
The quotient $\ell^j(n)\big/\langle\cR_n,\nu\rangle$ is homeomorphic to:
\begin{enumerate}
    \item $\mathscr{L}(j)$ for odd $k$.
    \item $\mathscr{L}(j)\cup L_{2j}(1,3,5,\dots,j-1)$ for even $k$ and even $j$.
    \item $\mathscr{L}(j)\cup\big( \cC_{L_j(1,3,\dots,j-2)}\big/\!\{([X],1)\sim([-X],1)\}\!\big)$, where $[X]\in L_j(1,3,\dots,j-2)$ denotes the class of $X\in \ell_{-}^{\,j}(n)$ with last coordinate equal to $0$, for even $k$ and odd $j$.
\end{enumerate}
\end{thm}

\begin{proof}The action of $\langle\cR_n,\nu\rangle$ in the sphere $\ell_{+}^{\,j}(n)\cong\S^{j-2}$ acts as the action of $\langle\cR_j,\nu\rangle$ (see Theorem \ref{esferas}). Then the corresponding quotient is $\ell_{+}^{\,j}(n)/\langle\mathcal{R}_{n},\nu\rangle\cong\mathscr{L}(j)$ (see Theorem \ref{quotient1}).

For even $k$ we have that $\ell_{-}^{\,j}(n)$ is a copy of $\S^{j-1}$. Since $\mathcal{R}^{j}_{n}=\nu$ restricted to $\ell_{-}^{\,j}(n)$, we have that $\ell_{-}^{\,j}(n)/\langle\mathcal{R}_{n},\nu\rangle\cong \ell_{-}^{\,j}(n)/\langle\mathcal{R}_{n}\rangle$. Notice that $\langle\cR_n\rangle$ restricted to $\ell_{-}^{\,j}(n)$ has order $2j$. We consider two cases:

\noindent
$\bullet~~$\textit{Case 1, j is even}.

If we identify the real subspace of dimension $j$ that contains $\ell_{-}^{\,j}(n)
$ with $\C^{\frac{j}{2}}$, the action of $\langle\cR_n\rangle$ corresponds to the action $$(z_1,z_2,\cdots,z_{\frac{j}{2}})\mapsto\big(e^{2\pi i(\frac{1}{2j})}z_1,e^{2\pi i(\frac{3}{2j})}z_2,\cdots,e^{2\pi i(\frac{j-1}{2j})}z_{\frac{j}{2}}\big).$$ Then $\ell_{-}^{\,j}(n)/\langle \mathcal{R}_{n}\rangle \cong L_{2j}(1,3,5,\dots,j-1)$ (see Proposition \ref{impar}).

\noindent
$\bullet~~$\textit{Case 2, if j is odd}.

Since $\cR_n^j=\nu$, then the upper hemisphere has representatives of any class of $\ell_{-}^{\,j}(n)/\langle\mathcal{R}_{n}\rangle$. As shown in the proof of Theorem \ref{esferas}, the $(n-1,n-1)$ entry of the matrix $\cR_n$ is $-1$. Let $S_r\subset \ell_{-}^{\,j}(n)$ be the set of points with the last coordinate equal to $r\in[0,1]$. There are three cases:

\begin{enumerate}
    \item The blocks of $\mathcal{R}_{n}$ that act in the equator $S_0$ corresponds to the blocks $R_{m\cdot\frac{2\pi}{k}}$ for $m\in \{k,3k,\dots,\frac{(j-2)k}{2}\}$ (see the proof of Theorem \ref{esferas}), then $S_0\big/\!\langle\cR_n\rangle$ is homeomorphic to the space $L_{2j}(1,3,\dots,j-2)$. 
    
    \item If $0<r<1$, the action of $\langle\cR_n\rangle$ in $S_r$ is determined by the action of $\langle\cR^2_n\rangle$ (the odd powers of $\cR_n$ send $S_r$ in $-S_r$, which is contained in the lower hemisphere of $\S^{j-1}$), then $S_r\big/\!\langle\cR_n\rangle=S_r\big/\!\langle\cR^2_n\rangle$ is homeomorphic to $L_j(1,3,\dots,j-2)$ (see the proof of Theorem \ref{esferas}). 

    \item Since $S_1=\{(0,0,\dots,0,1)\}$, then $S_{1}/\langle \mathcal{R}_{n}\rangle$ is a point.
\end{enumerate}

For any $m$, $2m=j+l$ for some odd $l$, then $\cR_n^{2m}(X)=\cR_n^{j+l}(X)=-\cR_n^l(X)$. Therefore, $L_{2j}(1,3,\dots,j-2)$ is homeomorphic to $L_j(1,3,\dots,j-2)\big/\{[X]\sim[-X]\}$. We conclude that $\ell_{-}^{\,j}(n) \big/\langle\cR_n\rangle$ is homeomorphic to $ \cC_{L_j(1,3,\dots,j-2)}\big/\!\{([X],0)\sim([-X],0)\}$.\end{proof}

\vspace{-0.3cm}
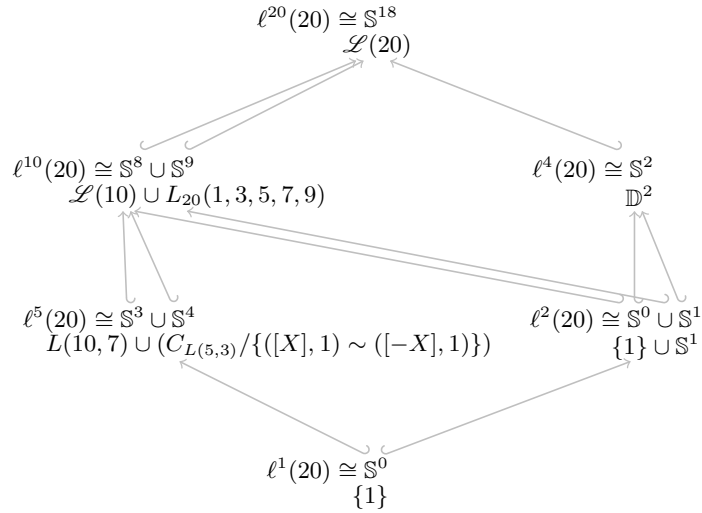
\begin{figure}[h]
\captionsetup{width=.8\linewidth}
\begin{center}\begin{tikzpicture}
\draw [<-right hook,semithick,gray!50] (-1.65,-0.55)--(-4.5,-1.75);
\draw [<-right hook,semithick,gray!50] (-1.55,-0.55)--(-3.9,-1.75);
\draw [<-left hook,semithick,gray!50] (-1.2,-0.55)--(1.86,-1.75);
\draw [<-left hook,semithick,gray!50] (2.05,-2.55)--(2.05,-3.78);
\draw [<-left hook,semithick,gray!50] (2.15,-2.55)--(2.65,-3.78);
\draw [<-left hook,semithick,gray!50] (-4.6,-2.55)--(1.9,-3.78);
\draw [<-left hook,semithick,gray!50] (-4.65,-2.55)--(-4.1,-3.78);
\draw [<-left hook,semithick,gray!50] (-4.75,-2.55)--(-4.7,-3.78);
\draw [<-left hook,semithick,gray!50] (-3.9,-2.55)--(2.5,-3.78);
\draw [<-left hook,semithick,gray!50] (-4,-4.55)--(-1.5,-5.8);
\draw [<-right hook,semithick,gray!50] (2,-4.55)--(-1.3,-5.8);

\draw (-2.03,0) node {\footnotesize $\ell^{20}(20)\cong \mathbb{S}^{18}$};
\draw (-1.37,-0.35) node {\footnotesize$\mathscr{L}(20)$};
\draw (-5,-2) node {\footnotesize$\ell^{10}(20)\cong \mathbb{S}^{8}\cup\mathbb{S}^{9}$};
\draw (-3.76,-2.35) node {\footnotesize$\mathscr{L}(10)\cup L_{20}(1,3,5,7,9)$};
\draw (1.5,-2) node {\footnotesize$\ell^{4}(20)\cong \mathbb{S}^{2}$};
\draw (2.12,-2.35) node {\footnotesize$\mathbb{D}^{2}$};
\draw (-4.93,-4) node {\footnotesize$\ell^{5}(20)\cong \mathbb{S}^{3}\cup\mathbb{S}^{4}$};
\draw (-2.81,-4.35) node {\footnotesize$L(10,7)\cup(C_{L(5,3)}/\{([X],1)\sim([-X],1)\})$};
\draw (1.81,-4) node {\footnotesize$\ell^{2}(20)\cong \mathbb{S}^{0}\cup \mathbb{S}^{1}$};
\draw (2.33,-4.35) node {\footnotesize$\{1\}\cup \mathbb{S}^{1}$};
\draw (-2.03,-6) node {\footnotesize $\ell^{1}(20)\cong \mathbb{S}^{0}$};
\draw (-1.46,-6.35) node {\footnotesize$\{1\}$};

\end{tikzpicture}\end{center}
\vspace{-0.3cm}
\caption{\footnotesize The sets of $20$-segments fixed by powers $j$ that are divisors of $20$. The first row in each node of the diagram shows the space $\ell^{j}(20)$, the second row shows the quotient $\ell^{j}(20)\big/\langle\cR_n,\nu\rangle$, and the lines indicate where each space is contained. We use that $L_5(1,3)=L(5,3)$, $\mathscr{L}(4)=\D^2$ and $\mathscr{L}(5)=L(10,7)$ (see Example \ref{eje3}).} 
\label{Figure5}
\end{figure}

Notice that if $m$ is a divisor of $j$, then $\ell^m(n)\subset\ell^j(n)$ and  $\ell^m(n)\big/\langle\cR_n,\nu\rangle\subset\ell^j(n)\big/\langle\cR_n,\nu\rangle$. Figure \ref{Figure5} shows the diagram of these spaces in the case $n=20$.

\section*{Acknowledgments}

This work was funded by CONAHCYT.

\bibliographystyle{amsalpha}

\bibliography{Bibliografia}

\end{document}